\documentclass[12pt]{amsart}
\usepackage{latexsym,amsxtra,amscd,ifthen}
\usepackage{cleveref}
\usepackage{amsfonts}
\usepackage{color,graphicx}
\usepackage{verbatim}
\usepackage{amsmath}
\usepackage{amsthm}
\usepackage{amssymb}
\usepackage{url}
\usepackage{xparse}
\usepackage{geometry}

\usepackage{lipsum}  
\usepackage[normalem]{ulem}
\usepackage{xcolor}
\makeatletter
    \newcommand{\culine}[2]{%
        \newcommand\temp@uline{\bgroup\markoverwith
            {\textcolor{#1}{\rule[-0.5ex]{2pt}{0.4pt}}}\ULon}%
        \temp@uline{#2}%
    }
    \newcommand{\cuuline}[2]{%
        \UL@protected\def\temp@uuline{\leavevmode \bgroup
            \UL@setULdepth
            \ifx\UL@on\UL@onin \advance\ULdepth2.8\p@\fi
            \markoverwith{\textcolor{#1}{\lower\ULdepth\hbox
                {\kern-.03em\vbox{\hrule width.2em\kern1\p@\hrule}\kern-.03em}}}%
        \ULon}%
        \temp@uuline{#2}%
    }
    \newcommand{\cuwave}[2]{%
        \UL@protected\def\temp@uwave{\leavevmode \bgroup
        \ifdim \ULdepth=\maxdimen \ULdepth 3.5\p@
        \else \advance\ULdepth2\p@
        \fi \markoverwith{\textcolor{#1}{\lower\ULdepth\hbox{\sixly \char58}}}\ULon}
        \font\sixly=lasy6 
        \temp@uwave{#2}%
    }
    \newcommand{\csout}[2]{%
        \UL@protected\def\temp@sout{\leavevmode \bgroup
        \ifdim \ULdepth=\maxdimen \ULdepth 3.5\p@
        \else \advance\ULdepth2\p@
        \fi \markoverwith{\textcolor{#1}{\lower\ULdepth\hbox{\sixly \char58}}}\ULon}
        \font\sixly=lasy6 
        \temp@sout{#2}%
    }
    \newcommand{\cxout}[2]{%
        \UL@protected\def\temp@uwave{\leavevmode \bgroup
        \ifdim \ULdepth=\maxdimen \ULdepth 3.5\p@
        \else \advance\ULdepth2\p@
        \fi \markoverwith{\textcolor{#1}{\lower\ULdepth\hbox{\sixly \char58}}}\ULon}
        \font\sixly=lasy6 
        \temp@uwave{#2}%
    }
    \newcommand{\cdashuline}[2]{%
        \UL@protected\def\temp@uwave{\leavevmode \bgroup
        \ifdim \ULdepth=\maxdimen \ULdepth 3.5\p@
        \else \advance\ULdepth2\p@
        \fi \markoverwith{\textcolor{#1}{\lower\ULdepth\hbox{\sixly \char58}}}\ULon}
        \font\sixly=lasy6 
        \temp@uwave{#2}%
    }
  \newcommand{\cdotuline}[2]{%
  \UL@protected\def\cdotuline{\leavevmode \bgroup 
    \UL@setULdepth
    \ifx\UL@on\UL@onin \advance\ULdepth2\p@\fi
    \markoverwith{\begingroup
       \lower\ULdepth\hbox{\kern.06em \textcolor{#1}{.}\kern.04em}%
       \endgroup}%
    \ULon}
    }
  \UL@protected\def\greendashuline{\leavevmode \bgroup 
    \UL@setULdepth
    \ifx\UL@on\UL@onin \advance\ULdepth2\p@\fi
    \markoverwith{\kern.13em
    \vtop{\color{green}\kern\ULdepth \hrule width .3em}%
    \kern.13em}\ULon}
\makeatother

\usepackage{tikz}
\usetikzlibrary{arrows,shapes,chains}

\theoremstyle{plain}

\newtheorem{theorem}{Theorem}
\newtheorem{lemma}[theorem]{Lemma}
\newtheorem{proposition}[theorem]{Proposition}

\newtheorem{corollary}[theorem]{Corollary}

\numberwithin{theorem}{section}
\numberwithin{equation}{theorem}

\theoremstyle{definition}
\newtheorem{definition}[theorem]{Definition}
\newtheorem{Notation}[theorem]{Notation}

\newtheorem{remark}[theorem]{Remark}

\newtheorem{question}[theorem]{Question}
\newtheorem*{question*}{Question}

\DeclareMathOperator{\Aut}{Aut}

\newcommand\kk{{\Bbbk}}

\def\Aut {\mathop{\rm Aut}\nolimits}

\begin{document}
\title{Aut-stable subspaces of Grassmann algebras}

\author{Mithat Konuralp Demir and Zahra Nazemian }

\address{Konuralp Demir: Middle East Technical University}

\email{konuralp.demir@metu.edu.tr}

\address{Nazemian: University of Graz, Heinrichstrasse 36, 8010 Graz, Austria}

\email{zahra.nazemian@uni-graz.at}





\begin{abstract}
Recently, the concept of $\Aut$-stable subspaces has played an important role in the characterization of polynomial rings, a topic that remains a challenging problem in algebraic geometry (see \cite{Kr-1996}). It turns out that polynomial rings with more than two variables do not have any $\Aut$-stable subspaces over an algebraically closed field of characteristic zero \cite{ZJLM}. 
 In this work, we characterize all $\Aut$-stable subspaces and $\Aut$-stable subalgebras of Grassmann  algebras.
\end{abstract}

\maketitle

In his 
Bourbaki seminar, Kraft \cite{Kr-1996} 
considered the following problem as one of 
the eight challenging problems in affine 
algebraic geometry (along with Jacobian 
Conjecture, Automorphism problem, the 
ZCP, etc):

\begin{question}  
\label{xxque0.10}
Find an algebraic-geometric characterization
of $\Bbb{C}[z_1,\cdots,z_m]$.
\end{question}

Recently an answer to this problem is given by the second author together with H. Huang, Y. Wang and J. J. Zhang. They have used the concept of $\Aut$-stable subspaces in this regard. Recall from 
\cite{ZJLM} that if $\kk$ is a field and $A$ is $\kk$-algebra, then 
a $\kk$-subspace $V$ of $A$ is called  
{\sf $\Aut$-stable} if $\sigma(V)\subseteq V$ 
for all algebra automorphisms $\sigma\in 
\Aut_{alg}(A)$. Recall that an algebra automorphism is a ring homomorphism from $A$ to $A$, which is isomorphism and $\kk$-linear. 
Every algebra $A$ has three 
obvious $\Aut$-stable subspaces, namely, $A$, 
$0$, and $\kk$, which are called {\sf trivial} 
$\Aut$-stable subspaces of $A$.

It is shown in \cite{ZJLM} that:

\begin{theorem} \label{lovelytheorem} \cite[Theorem 3.6]{ZJLM}
Suppose $\Bbbk$ is an algebraically closed 
field of characteristic zero and $A$ is a
$\kk$-algebra. Then $A$ is isomorphic to 
$\Bbbk[z_1,\cdots,z_m]$ for some integer $m\geq 2$ 
if and only if the following two conditions hold.
\begin{enumerate}
\item[(1)]
$A\neq \Bbbk$ is affine and connected graded,
\medskip
\item[(2)]
$A$ has no nontrivial $\Aut$-stable subspace.
\end{enumerate}
   
\end{theorem}

Moreover, it is shown that although  $ \kk[x]$, the polynomial ring in one variable over $\kk$,   has infinitely many $\Aut$-stable subspaces, it does not have any $\Aut$-stable subalgebras.
Every polynomial ring in $n$ variables is a quotient of the free algebra generated by $n$ variables, subject to the relations that the generators commute.
Understanding invariant ideals of a free algebra $A$---that is, ideals containing their image under any endomorphism of $A$---is a challenging problem in invariant theory; see, for example, \cite{BUL}.

Grassmann algebras are defined in a way similar to polynomial rings, except that the generators satisfy anti-commutativity rather than commutativity.

In more details, if $\kk$ is a field, a  {\sf Grassmann}  $\kk$-algebra generated by $n$ variables, denoted by $\mathcal{E}:= \bigwedge\nolimits_{\kk}(e_1, \dots, e_n)
$,
is the algebra generated by $e_1, \dots, e_n$ with the relations $e_i e_j = -e_j e_i$, for every $1 \leq i, j \leq n$. In particular, $e_i^2 = 0$, and so unlike the commutative polynomial ring, this algebra is never a domain.
The product in a Grassmann algebra is denoted by $\wedge $, as is standard in differential geometry. 

It is known (see also Definition \ref{defgraded} and Remark \ref{defgradedremark}) that both free algebras and Grassmann algebras are connected graded. Therefore, based on Theorem \ref{lovelytheorem}, they admit nontrivial 
$\Aut$-stable subspaces, see also Proposition \ref{zero}. 

Based on the definition of $\Aut$-stable subspaces, understanding them is directly connected to understanding all automorphisms of the algebras. It is valuable to mention that understanding all automorphisms of polynomial algebras with   three  or more variables is also listed among the eight challenging problems in algebraic geometry, see \cite{Kr-1996}.  
However, the authors of \cite{ZJLM} develop some techniques and use certain automorphisms of a polynomial rings
 to show that, when $m \geq 2$, $\kk[z_1, \dots, z_m]$ does  not have any nontrivial $\Aut$-stable subspaces.  
For the Grassmann algebra generated by $n$ variables, when $n$ is finite, the automorphisms are completely understood and studied in the groundbreaking work of D. {\v{Z}}. Djokovi{\'c}, see \cite{russ}.
An attempt to understand the automorphisms of the Grassmann algebra generated by countably infinitely many elements was made in \cite{notfinitecase, 2, 3}. 
A complete characterization has not yet been achieved, as  A. Guimar{\~a}es  mentioned in his recent talk at the European Non-Associative Algebra Seminar in 2025.\\

\emph{From now on, $\kk$ denotes a field of 
characteristic different from $2$, and all 
algebras are assumed to be associative with 
unity. $\mathcal{E}$ denotes the Grassmann 
$\Bbbk$-algebra generated by $n$ variables 
$ E := \{e_1, \cdots, e_n\}$, 
where $n \geq 1$ is finite. We consider the {\sf standard} grading on $\mathcal{E}$, that is, $\mathcal{E}_i$ is the $\kk$-vector space spanned by all products of $i$ distinct elements of $E$.   }

\bigskip

Based on this convention, we prove in this manuscript that:

\begin{theorem}
A nonzero subspace $B \subseteq \mathcal{E}$ is $\Aut$-stable if and only if $B$ has one of the following forms:
\begin{itemize}
    \item[(a)] For some even $0 <  j \leq n   $,  
    \[
   B=  \bigoplus_{\substack{j \leq  k \leq  n \\ k \ \text{even}}} \mathcal{E}_{k}.
    \]
\medskip

    \item[(b)] There exist  $1 \leq j \leq n$, such that  
    \[
    B =  \bigoplus_{\substack{j\leq  i} } \mathcal{E}_{i}.
    \]

\medskip

    \item[(c)] $B = B_1 + B_2$, where $B_1 $ is  of the form in (a) and $B_2 $  is of the form in  (b).  
    
\medskip

    \item[(d)] $B = \kk + B'$, where $B' = 0$ or $B'$ is of the form in (a) ,  (b) or (c).  
\end{itemize}
\end{theorem}

\bigskip
For subalgebras, we obtain the following result:

\begin{theorem}
A nonzero subalgebra $B \subseteq \mathcal{E}$ is $\Aut$-stable if and only if $B = \kk + B'$, where $B'$ is  $\Aut$-stable subspace of $\mathcal{E}$. 
\end{theorem}

\section{Basic Notions and  results}
In this section, $A$ denotes a $\kk$-algebra, where $\kk$ is any field, and $A$ is assumed to be associative with a unit. We recall the following definition.

\begin{definition}\label{defgraded}
A $\kk$-algebra $A$ is called {\sf connected graded} if 
\[
A = \bigoplus_{\substack{i \geq 0 \\ i \in \mathbb{Z}}} A_i,
\]

where $A_i$ are $\kk$-vector subspaces of $A$, $A_0 = \kk$, and $A_i A_j \subseteq A_{i + j}$ for all $i,j \geq 0$.
\end{definition}

\begin{remark}\label{defgradedremark}
If an algebra is connected graded, it may admit different gradings. However, a \emph{standard grading} on $A$, when $A$ is 
a commutative polynomial $\kk$-algebra, an associative free algebra, or a Grassmann algebra  generated by $n$ variables $\{x_1, \dots, x_n\}$—is given by letting 
\[
A_i := \operatorname{span}_{\kk} \left\{ x_{j_1}^{d_1} \cdots x_{j_t}^{d_t} \;\middle|\; 1 \leq j_k \leq n,\ \sum_{i'=1}^t d_{i'} = i \right\}.
\]
We use this grading throughout this work.
\end{remark}

The following result is taken from \cite{ZJLM}; however, we include a proof here for the convenience of the reader. Recall that when $a, b \in A$, the \emph{commutator} of $a$ and $b$ is defined as  
\[
[a, b] := ab - ba.
\]

\begin{proposition} \label{zero} \cite[Example 3.2.]{ZJLM}
Let $A $ be a connected graded noncommutative $\kk$-algebra. Then the subalgebra generated by commutators in $A$, denoted by $A_{\mathrm{com}}$, is a nontrivial $\Aut$-stable subalgebra. 
\end{proposition}

\begin{proof}
Let 
\[
G := \{ [a, b]  \mid a, b \in A \}.
\] 
Since $A$ is noncommutative, there exist $a, b \in A$ such that $ab \neq ba$, hence $[a, b] \neq 0 \in G$.  

Let $A_{\mathrm{com}}$ be the algebra generated by $G$, i.e., the intersection of all subalgebras of $A$ containing $G$. Equivalently, $A_{\mathrm{com}}$ consists of all $\kk$-linear combinations of products of elements in $G \cup \{1\}$.  

Clearly, $A_{\mathrm{com}}$ is $\Aut$-stable. Indeed, for every $f \in \Aut_{alg}(A)$, we have 
\[
f([a, b]) = [f(a), f(b)].
\]  

It remains to show that $A_{\mathrm{com}}$ is nontrivial. Take a grading on $A$, say  $A = A_0 \oplus A_1 \oplus A_2 \oplus \dots$ and suppose $A_1 \neq 0$. If 
\[
a = a_0 + a_1 + \cdots + a_n \quad \text{and} \quad b = b_0 + b_1 + \cdots + b_m,
\] 
with $a_i, b_i \in A_i$, then, since $A_0 = \kk$, we have  
\[
ab - ba \in A_2 \oplus A_3 \oplus \cdots.
\] 
Thus 
\[
G \subseteq A_2 \oplus A_3 \oplus \cdots,
\] 
and therefore 
\[
A_{\mathrm{com}} \subseteq \kk \oplus A_2 \oplus A_3 \oplus \cdots \subsetneq A,
\] 
which shows that $A_{\mathrm{com}}$ is indeed nontrivial.
\end{proof}

\bigskip
One method of constructing automorphisms of an algebra is via locally nilpotent derivations.
We recall the following definition.

\begin{definition}
Let \( A \) be a \( \kk \)-algebra.
\begin{enumerate}
    \item A map \( \partial : A \to A \) is {\sf \( \kk \)-linear} if  
    \[
    \partial(a + b) = \partial(a) + \partial(b), \quad \partial(ka) = k\,\partial(a)
    \]
    for all \( a, b \in A \), \( k \in \kk \).
\bigskip 
    \item A \( \kk  \)-linear map is called a {\sf derivation} if for every $a, b \in A$, we have: 
    \[
    \partial(ab) = \partial(a)b + a\partial(b).
    \]
    A derivation is {\sf locally nilpotent} if, for every \( a \in A \), there exists \( n \geq 1 \) such that \( \partial^n(a) = 0 \).
    It is easy to see that for every  $ a \in A$,
    the map
\[
[a, -] : A \to A, \quad x \mapsto ax - xa,
\]
  is a derivation of $A$ which is not in general a locally nilpotent one. 
    \bigskip 
\item  (See \cite{Fr-2006}) Given a locally nilpotent derivation $\partial$ of $A$ and  { $k\in \Bbbk\setminus\{0\}$}, we can define an automorphism of $A$ defined  by
\begin{equation}
\label{E1.11.2}\tag{E1.11.2}
\exp(k\partial)(a):=\sum_{i=0}^{\infty} 
\frac{k^i}{i!} \partial^i(a) 
\qquad {\text{for all $a\in A$}}.
\end{equation}
    
\end{enumerate}
\end{definition}

\begin{Notation}
  \label{not1}
Let $\mathcal{E}$ be a Grassmann algebra generated by 
\[
E := \{ e_1, \dots, e_n \}.
\]
We consider the {\sf standard} grading on $\mathcal{E}$, that is, $\mathcal{E}_i$ is the $\kk$-vector space spanned by all products of $i$ distinct elements of $E$.  

Because of the relations
\[
e_i \wedge e_j = -\, e_j \wedge e_i \quad \text{and} \quad e_i \wedge e_i = 0,
\]
a basis for $\mathcal{E}_i$, denoted by $E_i$,  consists of elements of the form
\[
e_{j_1} \wedge \cdots \wedge e_{j_i}
\quad \text{with} \quad 1 \leq j_1 < j_2 < \cdots < j_i \leq n.
\]
Therefore, for $1 \leq i \leq n$ we have
\[
\mathcal{E}_i = \mathrm{span}_{\kk} \left\{ e_{j_1} \wedge \cdots \wedge e_{j_i} \;\middle|\; 1 \leq j_1 < j_2 < \cdots < j_i \leq n \right\} = \mathrm{span}_{\kk} E_i. 
\]
Clearly,
\[
\mathcal{E}_n = \kk \, e_1 \wedge \cdots \wedge e_n.
\]

For $m \in \mathbb{Q}$, we write $\lfloor m \rfloor$ for the greatest integer less than or equal to $m$, and  
$\lceil m \rceil$ for the least integer greater than or equal to $m$.

The \emph{even part} of $\mathcal{E}$ is defined as  
\[
\mathcal{E}^{\mathrm{ev}} := \bigoplus_{i=0}^{\lfloor n/2 \rfloor} \mathcal{E}_{2i} = \bigoplus_{\substack{0 \leq i \leq n \\ i \ \text{even}}} \mathcal{E}_{i},
\]
and the \emph{odd part} of $\mathcal{E}$ is defined as  
\[
\mathcal{E}^{\mathrm{odd}} := \bigoplus_{i=1}^{\lceil n/2 \rceil} \mathcal{E}_{2i-1} = \bigoplus_{\substack{0 \leq i \leq n \\ i \ \text{odd}}} \mathcal{E}_{i}.
\]

Let $a = \sum_{i=0}^n a_i \in \mathcal{E}$, where $a_i \in \mathcal{E}_i$.  
The {\sf{even part}} of $a$ is defined as
\[
a_{\mathrm{even}} := \sum_{\substack{0 \leq i \leq n \\ i \ \text{even}}} a_i,
\]
and the {\sf {odd part}} of $a$ is defined as
\[
a_{\mathrm{odd}} := \sum_{\substack{0 \leq i \leq n \\ i \ \text{odd}}} a_i.
\]

\end{Notation}

\bigskip

For any $\kk$-algebra $A$, its center, denoted by $Z(A)$, is clearly an $\Aut$-stable subspace, which may be equal to either $\kk$ or $A$. 
\textit{From now on, by $\mathcal{E}$ we mean the Grassmann algebra generated by $
\{ e_1, \dots, e_n \},$
where $n$ is finite, and we use the standard grading as in Notation~\ref{not1}.}\\ 
In what follows, we compute the center of a Grassmann algebra.

\begin{lemma}
Let $\mathcal{E}$ be the Grassmann algebra generated by $n$ elements. Then  
\[
Z(\mathcal{E}) =  \mathcal{E}^{\mathrm{ev}} \; + \; \mathcal{E}_n .
\]
\end{lemma}

\begin{proof}
Clearly, every element of $\mathcal{E}_{2i}$, when $ 0 \leq  i \leq  \lfloor n/2 \rfloor $ lies in the center. Moreover, if $T \in \mathcal{E}_n$, then $T \wedge a = 0 = a \wedge T$ for every $a \in \bigoplus _{i = 1} ^n \mathcal{E}_i $ and    $T \wedge a =  a \wedge T$, when $a \in \kk$. 
This shows that 
\[
\bigoplus_{i=0}^{\lfloor n/2 \rfloor} \mathcal{E}_{2i}  \; + \; \mathcal{E}_n \subseteq Z(\mathcal{E}).
\]

Conversely, let $x \in Z(\mathcal{E})$. Write
\[
x = x_1 + x_2,
\]
where $x_1 \in \bigoplus_{i=0}^{\lfloor n/2 \rfloor} \mathcal{E}_{2i} + \mathcal{E}_n$ and   
$x_2 \in   \bigoplus_{i=0}^{\lfloor n/2 \rfloor - 1} \mathcal{E}_{2i + 1}$.

Since $x, x_1 \in Z(\mathcal{E})$, it follows that $x_2 \in Z(\mathcal{E})$.  
We can write
\[
x_2 = \sum_{i=0}^{\lfloor n/2 \rfloor - 1} a_{2i+1}, \quad a_{2i+1} \in \mathcal{E}_{2i+1}.
\]
We claim that $x_2 = 0$. Suppose not. Then there exists  $ 0\leq j \leq \lfloor n/2 \rfloor - 1 $ such that $a_{2j +1} \neq 0$. Since ${2j +1} < n$, there exists $e_k \in E$ such that $a_{2j + 1} \wedge e_k \neq 0$.  

For every
$ 0\leq i \leq \lfloor n/2 \rfloor - 1 $, since $2i + 1$ is odd, 
 we have either $a_{2i+1} \wedge e_k = 0 = e_k \wedge a_{2i+1}$ or $a_{2i+1} \wedge e_k = -\, e_k \wedge a_{2i+1}$.  
Therefore, 
\[
 0 \neq x_2 \wedge e_k = -\, e_k \wedge x_2,
\]
which contradicts the fact that $x_2 \in Z(\mathcal{E})$.  
Hence $x_2 = 0$, and the claim follows.
\end{proof}

\bigskip

Now we compute 
$\mathcal{E}_{\mathrm{com}}$,  defined in Proposition~\ref{zero}:

\begin{lemma} \label{com}
Let $\mathcal{E}$ be the Grassmann algebra generated by $n$ elements. Then  
\[
   \mathcal{E}_{\mathrm{com}} = \mathcal{E}^{\mathrm{ev}}. 
\]
\end{lemma}

\begin{proof}
Since $\mathcal{E}_{\mathrm{com}}$ is a subalgebra and contains $\kk$, to prove that 
$\mathcal{E}^{\mathrm{ev}} \subseteq \mathcal{E}_{\mathrm{com}}$, it is enough to show that for every nonzero even integer $m$,  the basis of 
$\mathcal{E}_m$, namely $E_m$, are contained in $\mathcal{E}_{\mathrm{com}}$.  
Let 
\[
    e_{i_1} \wedge \cdots \wedge e_{i_m} \in {E}_m,
\]
where $m \geq 2$ is even. Then  
\[
    e_{i_1} \wedge \cdots \wedge e_{i_m} = 
    \tfrac{1}{2} \big[\, e_{i_1} \wedge \cdots \wedge e_{i_{m -1}},\ e_{i_m} \,\big].
\]
Recalling the definition of $G$ as the set containing all commutators in the proof of Lemma~\ref{zero}, we have 
\[
   \big[\, e_{i_1} \wedge \cdots \wedge e_{i_{m -1}},\ e_{i_m} \,\big] \in G,
\]
and consequently $ e_{i_1} \wedge \cdots \wedge e_{i_m} \in \mathcal{E}_{\mathrm{com}}$.  
Therefore $\mathcal{E}^{\mathrm{ev}} \subseteq \mathcal{E}_{\mathrm{com}}$.  

To prove the reverse inclusion, note that since $\mathcal{E}_{\mathrm{com}}$ is a subalgebra, it suffices to show that the elements of $G$ are in $\mathcal{E}^{\mathrm{ev}}$.  
Take $a, b \in \mathcal{E}$. We can write  
\[
    a = \sum_{i = 0}^n a_i, \qquad b = \sum_{i = 0}^n b_i,
\]
where $a_i, b_i \in \mathcal{E}_i$. We then have $a_i \wedge b_j = b_j \wedge a_i$ whenever $i$ or $j$ is even.  
Therefore, 
\[
    [a_i, b_j] = 0 \quad \text{whenever $i$ or $j$ is even}.
\]
Moreover, 
\[
    [a_i , b_j] \in \mathcal{E}^{\mathrm{ev}} \quad \text{whenever both $i$ and $j$ are odd}.
\]
This implies that $[a, b]$, which is equal to $ \sum_{0 \leq i, j \leq n} [a_i, a_j]$ is an element of $ \mathcal{E}^{\mathrm{ev}}$, and hence  
\[
   \mathcal{E}^{\mathrm{ev}} = \mathcal{E}_{\mathrm{com}}. 
\]
\end{proof}

As a consequence of the above two lemmas, when $n$ is even we have  $
   \mathcal{E}_{\mathrm{com}} = Z(\mathcal{E}).
$

\section{$\Aut$-stable subspaces
(subalgebras)
of $\mathcal{E}$
}
In this section, we characterize 
$\Aut$-stable subspaces and subalgebras of the Grassmann algebra $\mathcal{E}$. 
To achieve this goal, we first recall some properties of its automorphism group, $ \Aut_{{alg}}(\mathcal{E})$, 
which have been characterized and are well understood in \cite{russ}. 

A (not necessarily commutative) ring is called {\sf local} if the set of all non-invertible elements 
forms an ideal, or equivalently, if there exists a unique maximal right (equivalently, left) ideal.  

If $a \in M := \bigoplus_{i \geq 1} \mathcal{E}_i$,  
then $a^{\,n+1} = 0$, which means that $a$ is nilpotent. 
Since the sum of a nilpotent element and an invertible element is invertible, 
$\mathcal{E}$ is a local ring with maximal ideal $M$. 

\bigskip

\begin{lemma}[\cite{russ}, Lemma~7] \label{esyfact}
If $a, b \in \mathcal{E}^{\mathrm{odd}}$, then 
\[
   [a, [b, -]] = 0.
\]
In particular, $[a, [a, -]] = 0$. 
\end{lemma}

\bigskip

When $Id_{\mathcal{E}}$ denotes the identity map on ${\mathcal{E}}$, Lemma~\ref{esyfact} together with Equation~\eqref{E1.11.2} implies that if
$a \in \mathcal{E}^{\mathrm{odd}}$, then 
\[
   \exp([a, -]) = Id_{\mathcal{E}} + [a, -],
\]
which maps $x \in \mathcal{E}$ to $x + [a, x]$  defines an automorphism of $\mathcal{E}$.  

Consider 
\[
   N_1 := \{\, Id_{\mathcal{E}} + [a, -] \;\mid\; a \in \mathcal{E}^{\mathrm{odd}} \,\},
\]
which is an abelian subgroup of $\Aut_{\mathrm{alg}}(\mathcal{E})$. Indeed, by Lemma~\ref{esyfact}, for every $a, b \in \mathcal{E}^{\mathrm{odd}}$ we have
\[
   (Id_{\mathcal{E}} + [a, -])(Id_{\mathcal{E}} + [b, -]) 
   = Id_{\mathcal{E}} + [a, -] + [b, -] 
   = (Id_{\mathcal{E}} + [b, -])(Id_{\mathcal{E}} + [a, -]).
\]

Now define 
\[
   F_0 := \{\, f \in \Aut_{\mathrm{alg}}(\mathcal{E}) \;\mid\; f(\mathcal{E}^{\mathrm{ev}}) \subseteq \mathcal{E}^{\mathrm{ev}},\ 
   f(\mathcal{E}^{\mathrm{odd}}) \subseteq \mathcal{E}^{\mathrm{odd}} \,\}.
\]

If $V := \kk e_1 + \cdots + \kk e_n$ is the vector space spanned by the generators $e_1, \dots, e_n$, and $f$ is a $\kk$-linear map from $V$ to $V$, then $f$ extends uniquely to an automorphism of $\mathcal{E}$, which necessarily belongs to $F_0$.  
Moreover, since $Id_{\mathcal{E}} = Id_{\mathcal{E}} + [ 0, - ]$,   we have  that $N_1 \cap F_0 = \{ Id_{\mathcal{E}} \}$. Using the fact that $\mathcal{E}$ is local, D.~{\v{Z}}.~Djokovi{\'c} shows in \cite{russ} that:

\bigskip

\begin{theorem} [\cite{russ}, Theorem~9] \label{es1}
Let $\mathcal{E}$ be the Grassmann algebra generated by $n$ elements. Then
\[
   \Aut_{\mathrm{alg}}(\mathcal{E}) = N_1 \rtimes F_0.
\]
\end{theorem}

\begin{definition}\label{basicdef}
\begin{itemize}
    \item  [(i)] 
Let $0 \neq k_0 \in \kk$, and let $x = e_{i_1} \wedge \cdots \wedge e_{i_m} \in E_m$, where $m \geq 1$,  be an 
element in the basis  of $\mathcal{E}_m$.  
Then $k_0x$ is called an {\sf irreducible element}.  

The {\sf support} of $k_0x$, denoted $\sup(k_0x)$, is the set
\[
   \sup(k_0x) := \{ e_{i_1}, \ldots, e_{i_m} \}.
\]

By an \textsf{irreducible decomposition} of $0 \neq b \in \mathcal{E}_m$, where $m \geq 1$, we mean a representation
\[
b = b_1 + \cdots + b_k
\]
such that  each $b_i$ is an irreducible element  and for every $1 \leq j \neq t \leq k$, we have $\sup(b_j) \neq \sup(b_t)$.  
Clearly, this decomposition is unique and we write
\[
\operatorname{Irr}(b) = k,      {} {} {} \sup (b) := \bigcup _{i = 1}^k \sup (b_i) 
\]

We define $\operatorname{Irr}(0) = 0$ and $\operatorname{Irr}(k_0) = 1$.   
More generally, for $a = \sum_{i=0}^n a_i$ with $a_i \in \mathcal{E}_i$, we set
\[
\operatorname{Irr}(a) := \sum_{i=0}^n \operatorname{Irr}(a_i),  {}{}{} \sup (a):= \bigcup _{i = 1}^n \sup (a_i).
\]

    \item [(ii)] 
For $1 \leq i \neq j \leq n$, we denote by $\phi_{i,j}$ the automorphism of $\mathcal{E}$ that interchanges 
$e_i$ and $e_j$, while fixing all other basis elements of $E$. Another automorphism that we will use is denoted by $\phi_S$, 
where $S \subseteq E$. It is defined on the basis elements by
\[
   \phi_S(e_i) = 
   \begin{cases}
      -e_i, & \text{if } e_i \in S, \\[6pt]
       e_i, & \text{if } e_i \notin S.
   \end{cases}
\]
\end{itemize}

\end{definition}

\bigskip
Clearly, we have:

\begin{lemma}
Let $x = e_{i_1} \wedge \cdots \wedge e_{i_m}$ be an irreducible element and let $e_i, e_j \in \sup(x)$ with $i \neq j$.  
Then 
\[
   \phi_{i,j}(x) = -x.
\]
\end{lemma}

\bigskip

{
We now begin to study $\Aut$-stable subspaces of the Grassmann algebra.  

\begin{proposition} \label{3.5}
Let $B$ be an $\Aut$-stable subspace of $\mathcal{E}$. 
Then the following hold:
\begin{itemize}
    \item[(i)] If $ \mathcal{E}_i  \cap B \neq 0$, then $\mathcal{E}_i \subseteq B$.  
    \medskip
    \item[(ii)] Let 
      $a = \sum_{i = 0}^{n} a_i \in B$, where $ a_i \in \mathcal{E}_i$. If  
    $a_j \neq 0$, then $\mathcal{E}_j \subseteq B$.  
     \medskip
     \item[(iii)] 
     Let 
      $a = \sum_{i = 0}^{n} a_i \in B$, where $ a_i \in \mathcal{E}_i$. Then 
     $\sum_{i = 0}^{n} [b, a_i] \in B,  \quad \text{for all } b \in \mathcal{E}^{\mathrm{odd}}$.  
   
\end{itemize}
\end{proposition}

\begin{proof}
(i) Since $\mathcal{E}_n$ is one-dimensional,  if  
$0 \neq \mathcal{E}_n \cap B$, then $\mathcal{E}_n \subseteq B$. Moreover, it is clear that if $ \kk \cap B \neq 0$, then 
$ \kk \subseteq  B$. 
Suppose that  $ 0 <i < n$, 
and  $ 0 \neq a \in \mathcal{E}_i \cap B$. 
We  write the {irreducible decomposition} of $a$
\[
a = a_{i_1} + \cdots + a_{i_m}.
\]
We proceed by induction on $m$.

If $m = 1$, then $B$ contains an element of ${E}_i$, and by applying a composition of automorphisms of the form  
$\phi_{j,j'}$, $ 1 \leq j \neq j' \leq n$,  we see that $B$ contains all basis elements of ${E}_i$, and consequently $\mathcal{E}_i \subseteq B$.  

Now suppose $m > 1$. Then there exists $e_j \in \sup(a_{i_m}) \setminus \sup(a_{i_{m-1}})$.  
Consider the automorphism $\phi_{\{e_j\}}$ introduced in Definition~\ref{basicdef}. Then 
\[
\phi_{\{e_j\}}(a_{i_m}) = -a_{i_m}, \quad \text{and} \quad \phi_{\{e_j\}}(a_{i_{m-1}}) = a_{i_{m-1}}.
\]  
This implies that $\phi_{\{e_j\}}(a) + a$, which is a nonzero element of $B$, is a sum of at most $m-1$ irreducible elements.  
By the induction hypothesis, we conclude that $\mathcal{E}_i \subseteq B$.

\medskip
(ii) Let 
\[
a = \sum_{i=0}^{n} a_i \in B, \qquad a_i \in \mathcal{E}_i, \quad a_j \neq 0.
\]
The proof is by induction on $j$.

\medskip
\noindent
\emph{Base case $j=0$.}  
Define
\[
S_0 := \left\{\, b = \sum_{i=0}^n b_i \in B \;\middle|\; b_i \in \mathcal{E}_i,\, b_0 \neq 0 \,\right\}.
\]
This set is nonempty, since $a \in S_0$.  
Choose $b' \in S_0$ such that $\operatorname{Irr}(b')$ is 
minimum in the set 
\[
\{\, \operatorname{Irr}(b) \mid b \in S_0 \,\}.
\]
We claim that $\operatorname{Irr}(b') = 1$.  
Suppose not, so that $\operatorname{Irr}(b') > 1$.  
Let $e_i \in \sup(b')$. Then 
\[
b' + \phi_{\{e_i\}}(b') \in S_0
\quad\text{and}\quad 
\operatorname{Irr}\!\left(b' + \phi_{\{e_i\}}(b')\right) < \operatorname{Irr}(b'),
\]
which contradicts the minimality of $\operatorname{Irr}(b')$.  
Therefore $\operatorname{Irr}(b') = 1$, which means that $b' \in \Bbbk$ is nonzero.  
Hence $\Bbbk = \mathcal{E}_0 \subseteq B$.

\medskip
\noindent
\emph{Induction step $j>0$.}  
By the induction hypothesis, if $\mathcal{E}$ contains an element whose $i$-th component is nonzero for some $i<j$, then $\mathcal{E}_i \subseteq \mathcal{E}$.  
This implies that 
\[
\sum_{i=j}^n a_i \in B,
\]
so the set
\[
S_j := \left\{\, b = \sum_{i=j}^n b_i \in B \;\middle|\; b_i \in \mathcal{E}_i,  b_j \neq 0 \,\right\}
\]
is nonempty.  

Choose 
\[
b' = \sum_{i=j}^n b'_i \in S_j
\]
such that 
\[
\operatorname{Irr}(b') - \operatorname{Irr}(b'_j)
\]
is minimal among all 
\[ \{
\operatorname{Irr}(b) - \operatorname{Irr}(b_j), \qquad b = \sum_{i=j}^n b_i \in S_j.\}
\]

We claim that
\[
\operatorname{Irr}(b') - \operatorname{Irr}(b'_j) = 0,
\]
which is equivalent to 
\[
\sum_{i=j+1}^n b'_i = 0.
\]
Suppose instead that 
\[
\operatorname{Irr}(b') - \operatorname{Irr}(b'_j) > 0.
\]
Then let $ b'_j = c_1 + \cdots + c_m $ be the irreducible decomposition of $b'_j$. There exists 
\[
e_k \in \sup\!\Bigl(\sum_{i=j+1}^n b'_i\Bigr) \setminus 
\sup ( c_1 )
\]
  
Set 
\[
b'' := b' + \phi_{\{e_k\}}(b') \in S_j.
\]
Then 
\[
\operatorname{Irr}(b'') - \operatorname{Irr}(b''_j) 
< \operatorname{Irr}(b') - \operatorname{Irr}(b'_j),
\]
 which is a contradiction. 

Therefore our claim holds, and hence $0 \neq b'_j  \in \mathcal{E}_j \cap B$.  
By part (i), this implies that 
$
\mathcal{E}_j \subseteq B.$
  
\medskip
    (iii) This holds because, as explained before Theorem~\ref{es1}, for every 
$b \in \mathcal{E}^{\mathrm{odd}}$ the map 
\[
\phi := Id_{\mathcal{E}} + [b, -]
\]
lies in $N_1 \subseteq \Aut_{\mathrm{alg}}(\mathcal{E})$. 
Consequently, for every 
$a = \sum_{i = 0}^{n} a_i \in B$, where $ a_i \in \mathcal{E}_i$, we have  
\[
\phi(a) - a = \sum_{i=0}^{n} [b, a_i] \in B.
\]

\end{proof}
}

As a consequence of Proposition \ref{3.5}, we obtain: 

\begin{corollary} \label{cor2.6}
Let $a = \sum_{i = 0}^{n} a_i \in B$, where $B$ is an $\Aut$-stable subspace of $\mathcal{E}$, and suppose that $a_j \neq 0$ for some $1 \leq j \leq n$. Then:
\begin{itemize}
    \item[(i)] If $j < n$ is odd, then $\mathcal{E}_{j+1} \subseteq B$. 
    \medskip
    \item[(ii)] If $ 0 < j < n-1$ is even, then $\mathcal{E}_{j+2k} \subseteq B$ for every $k \geq 0$. 
\end{itemize}
\end{corollary}

\begin{proof}
\begin{itemize}
  \item[(i)] By Part~(ii) of Proposition~\ref{3.5}, we know that $\mathcal{E}_j \subseteq B$. Since $j < n$, there exists some $e_k$ and an irreducible component of $a_j$ such that $e_k$ is not in its support. By Part~(iii) of Proposition~\ref{3.5}, 
  \[
   a' := [e_k , a] \in B.
  \]
  Since the $(j+1)$-st component of $a'$ is nonzero, Part~(ii) of Proposition~\ref{3.5} implies that $\mathcal{E}_{j+1} \subseteq B$.
\medskip

  \item[(ii)] Suppose that $j < n - 1$ is even. From Part~(ii) of Proposition~\ref{3.5}, we have $\mathcal{E}_j \subseteq B$, and hence 
  \[
  e_{1} \wedge \cdots \wedge e_j \in B.
  \]
  Consider the $\kk$-linear ring homomorphism $f$ on $\mathcal{E}$ defined by
  \[
  f(e_1) = e_1 + e_1 \wedge e_{j+1} \wedge e_{j+2}, \qquad 
  f(e_i) = e_i \ \text{ for } i \neq 1.
  \]
  Note that $f$ is well defined. Indeed, for every $1 \leq i', i'' \leq n$, we have
\[
f(e_{i'} \wedge e_{i''}) 
= - f( e_{i''} \wedge e_{i'}).
\]
Moreover, $f$ is an automorphism, since its inverse is given by
\[
f^{-1}(e_1) = e_1 - e_1 \wedge e_{j+1} \wedge e_{j+2}, 
\qquad
f^{-1}(e_i) = e_i \quad \text{for } i \neq 1.
\]
  Consequently,
  \[
  0 \neq f(e_{1} \wedge \cdots \wedge e_j) - e_{1} \wedge \cdots \wedge e_j 
  \in B \cap \mathcal{E}_{j+2}.
  \]
  Hence $\mathcal{E}_{j+2} \subseteq B$. Repeating this process shows that $\mathcal{E}_{j+2k} \subseteq B$ for all $k \geq 0$.
\end{itemize}
\end{proof}

\bigskip

We are now in a position to give a characterization of the $\Aut$-stable subspaces of $\mathcal{E}$.

\begin{theorem} 
A nonzero subspace $B \subseteq \mathcal{E}$ is $\Aut$-stable if and only if $B$ has one of the following forms:
\begin{itemize}
    \item[(a)] For some even $0 <  j \leq n   $,  
    \[
   B=  \bigoplus_{\substack{j \leq  k \leq  n \\ k \ \text{even}}} \mathcal{E}_{k}.
    \]
\medskip

    \item[(b)] There exist  $1 \leq j \leq n$, such that  
    \[
    B =  \bigoplus_{\substack{j\leq  i} } \mathcal{E}_{i}.
    \]

\medskip

    \item[(c)] $B = B_1 + B_2$, where $B_1 $ is  of the form in (a) and $B_2 $  is of the form in  (b).  
    
\medskip

    \item[(d)] $B = \kk + B'$, where $B' = 0$ or $B'$ is of the form in (a) ,  (b) or (c).  
\end{itemize}
\end{theorem}

\begin{proof}
($\implies$)  
We only need to  check that the subspaces of type (a) and (b) are indeed $\Aut$-stable.  
Let $f \in \Aut_{alg}(\mathcal{E})$. Note that $e_s \wedge e_s = 0$, for every $1 \leq  s \leq n$,  and since $f$ is a ring homomorphism, we have
\[
0 = f(0) = f(e_s) \wedge f(e_s).
\]
This implies that
\[
f(e_s) \in \bigoplus_{t=1}^{n} \mathcal{E}_t.
\]
Hence, 
\[
f(\mathcal{E}_s) \subseteq \bigoplus_{t=s}^{n} \mathcal{E}_t.
\]
Now, by Theorem~\ref{es1}, it follows that every subspace of type (a) or (b) is $\Aut$-stable.

\smallskip
($\impliedby$)  
Conversely, let $B$ be a nontrivial $\Aut$-stable subspace of $\mathcal{E}$.  
Define
\[
S := \{\, 0 < i \leq n \mid \mathcal{E}_i \subseteq B \,\} = S_1 \cup S_2,
\]
where $S_1$ is  a subset of  odd integers and $S_2$ of  even integers. Let $B$ does not contain $\kk$. 
If $S_1$ is empty, then by Corollary~\ref{cor2.6}(ii), $B$ is of type~(a).
Suppose that $S_1$ is nonempty and let $j = \min S_1$.  
Then, by Corollary~\ref{cor2.6}(i), we have $\mathcal{E}_{j+1} \subseteq B$.  
Also, when $n \ge 3$, the map that sends $e_1$ to 
$e_1 + e_1 \wedge e_2 \wedge e_3$ and sends every other $e_i$ (for $i \neq 1$) to $e_i$ shows that every odd element greater than $j$ lies in $S_1$, and consequently every even element greater than $j$ lies in $S_2$.  
Hence, $B$ is of type~(b) or sum of an element in form (a) and another one in form (b).  
Now, if $B$ contains $\kk$, then we have $B = \kk + B'$, where $B'$ is an Aut-stable subspace, and thus $B$ is of type~(d).

\end{proof}

\bigskip

 As a corollary, we obtain a complete description of the $\Aut$-stable  subalgebras of $\mathcal{E}$ as follows:

\begin{corollary}\label{prefin}
A nonzero subspace $B \subseteq \mathcal{E}$ is $\Aut$-stable if and only if $B = \kk + B'$, where $B'$ is one the followings: 
\begin{itemize}
    \item[(a)] For some even $0 <  j \leq n   $,  
    \[
   B'=  \bigoplus_{\substack{j \leq  k \leq  n \\ k \ \text{even}}} \mathcal{E}_{k}.
    \]
\medskip

    \item[(b)] There exist  $1 \leq j \leq n$, such that  
    \[
    B' =  \bigoplus_{\substack{j\leq  i} } \mathcal{E}_{i}.
    \]

\medskip

    \item[(c)] $B' = B_1 + B_2$, where $B_1 $ is  of the form in (a) and $B_2 $  is of the form in  (b).  
\end{itemize}
\end{corollary}

\bigskip

Automorphisms of Grassmann algebras with infinitely many variables have been studied by several researchers (see, for example, \cite{notfinitecase, 2, 3}). However, a complete characterization has not yet been achieved. In this context, we ask:

\begin{question}
    What are the $\Aut$-stable subspaces (subalgebras) of Grassmann algebras generated by countably infinitely many elements?
\end{question}

\subsection*{Acknowledgments.} 
Nazemian was supported by the Austrian Science Fund (FWF) under grant P 36742. She would like to thank Jian James Zhang for introducing her to the concept of invariant theory during joint work \cite{ZJLM}, and Vesselin Drensky
for his valuable comments on free algebras during a conference in Graz and also  Alan de Ara{\'u}jo Guimar{\~a}es for introducing some references. This manuscript was written during the DRP Türkiye 2025 program, and the authors would like to thank its organizers.

\bibliography{main}
\bibliographystyle{amsplain}

 \end{document}